\documentclass[11pt,oneside,english]{amsart}
\usepackage{amssymb}
\usepackage[colorlinks]{hyperref}
\usepackage{graphicx}

\makeatletter \theoremstyle{plain}
 \newtheorem{thm}{Theorem}[section]
 \newtheorem{lem}[thm]{Lemma}
 \newtheorem{prop}[thm]{Proposition}
 
 \numberwithin{equation}{section} 
 \numberwithin{figure}{section} 
 \theoremstyle{plain}
 \theoremstyle{definition}

\newcommand{\calE}{{{\mathcal E}}}

\newcommand{\mH}{{{\mathbb H}}}

\newcommand{\calR}{{{\mathcal R}}}

\newcommand{\C}{{{\mathbb C}}}

\newcommand{\R}{{{\mathbb R}}}

\usepackage{babel}

\makeatother

\begin{document}

\title [Modulus of the Kor\'anyi ellipsoidal ring]
{The modulus of the Kor\'anyi ellipsoidal ring}

\author{Gaoshun Gou \& Ioannis D. Platis}

\address {Department of Mathematics,
Hunan-University,
Changsha {\rm 410082},
P. R. China.}
\email{gaoshungou@hnu.edu.cn}
\address{Department of Mathematics and Applied Mathematics,
University of Crete,
University Campus,
GR-70013 Heraklion Crete,
 Greece.}
\email{jplatis@math.uoc.gr}


\begin{abstract}
The Kor\'anyi ellipsoidal ring $\mathcal{E}$ of radii $B$ and $A$, $0<B<A$, is defined as the image of the Kor\'anyi spherical ring of the same radii  and centred at the origin via a linear contact map $L$ in the Heisenberg group. If $K\ge 1$ is the maximal distortion of $L$ then we prove that the modulus of $\mathcal{E}$ is equal to $$
{\rm mod}(\calE)=\left(\frac{3}{8}\Big(K^2+\frac{1}{K^2}\Big)+\frac{1}{4}\right)\frac{\pi^2}{(\log (A/B))^3}.
$$

\end{abstract}

\thanks{Part of this work has been carried out while IDP was visiting Hunan University, Changsha, PRC. Hospitality is gratefully appreciated. }

\maketitle

\section{Introduction}

In contrast to the complex plane case, conformal mappings in the Heisenberg group $\mathbb{H}^1$ comprise only the group of M\"obius transformations. This, together with the lack of analogues of Riemann Mapping and Steiner's Symmetrisation Theorems, require new techniques for the calculation of moduli of curve families in $\mathbb{H}^1$ since that there can be no standard normalisation for an arbitrary domain in $\mathbb{H}^1$. Recall that the Heisenberg group $\mathbb{H}^1$ is $\C\times\R$ with multiplication law
$$
(z,t)*(w,s)=(z+w,t+s+\Im(z\overline{w}),
$$
for each $(z,t),(w,s)\in\C\times\R$. It is a 2-step nilpotent Lie group; it is the model for both contact and sub-Riemannian geometry. There are two standard metrics in $\mathbb{H}^1$: The sub-Riemannian (Carnot-Carath\'eodory) metric ${\rm d}_{cc}$ and the Kor\'anyi metric ${\rm d}_{\mathbb{H}^1}$; the latter is not a path metric but it is bi-Lipschitz equivalent to ${\rm d}_{cc}$. 

Quasiconformal mappings between domains of $\mathbb{H}^1$ may be equivalently defined in terms of both ${\rm d}_{cc}$ and  ${\rm d}_{\mathbb{H}^1}$. The study of such mappings go back to the pioneering works of Kor\'anyi-Reimann, \cite{KR85,KR95} and Pansu, \cite{P89} and since then it has been generalised to more abstract spaces, see for instance \cite{He-KO98}. A primary goal in this area of study is to understand the similarities as well as the differences between the quasiconformal mapping theory in the sub-Riemannian and the Riemannian case, in particular the complex case where we have the powerful tool of Ahlfors-Bers theory. It is therefore important to study {\it extremal} quasiconformal mappings, that is, mappings with constant maximal distortion, between domains in $\mathbb{H}^1$. There is a long history in the study of extremal mappings in the complex plane. However, for the Heisenberg group case the picture is quite different: the lack of tools analogues to such as Teichm\"uller's Existence and Uniqueness Theorems turn the solution of problems of extremality into a quite difficult task. So far, calculation of moduli of curve families and in particular, capacities of condensers (see Section \ref{sec-prel} for the definitions) has been proved fundamental. For instance, in \cite{BFP13} the problem of finding the minimiser of a mean distortion integral of mappings belonging into a family of quasiconformal mappings between Kor\'anyi spherical rings is solved by using a modulus method; the solution for this problem is the radial stretch map which also minimizes the maximal distortion within the subclass of sphere-preserving mappings. 

A Kor\'anyi spherical ring is a domain bounded between two ${\rm d}_{{\mathbb H}^1}$ metric spheres of radii $B$ and $A$, $0<B<A$. Kor\'anyi and Reimann proved in \cite{KR87} that its modulus is equal to $\pi^2(\log(A/B))^{-3}$. Platis also showed in \cite{P} that the modulus of a {\it revolution ring}, that is, a domain bounded between two dilated images $D_B(\mathcal{S})$ and $D_A(\mathcal{S})$ of a surface of revolution $\mathcal{S}$, is also equal to $\pi^2(\log(A/B))^{-3}$.

In this paper we turn our attention into {\it Kor\'anyi ellipsoidal rings} in the Heisenberg group. These are images of Kor\'anyi spherical rings by linear contact quasiconformal maps. In Section \ref{sec-main} we prove the following: 

\medskip

\begin{thm}\label{thm1}
Let $\calR$ be a Kor\'anyi spherical ring centred at the origin and with radii  $B$ and $A$, $0<B<A$, and let the Kor\'anyi ellipsoidal ring $\calE=L_{a,b}(\calR)$ where $L_{a,b}$ is the contact linear map with distortion $K=a/b>1$. Then the modulus ${\rm mod}(\calE)$ is
$$
{\rm mod}(\calE)=\left(\frac{3}{8}\Big(K^2+\frac{1}{K^2}\Big)+\frac{1}{4}\right)\frac{\pi^2}{(\log (A/B))^3}.
$$
\end{thm}

\medskip

The method we use in order to prove this theorem is based on the fact that the modulus of the curve family of curves joining the two boundary pieces of a condenser equals to the capacity of the condenser, see Section \ref{sec-modcap}. Proposition \ref{pro1} which establishes this result is the powerful tool behind our rather simple method of calculating the modulus of $\calE$. We finally remark that Theorem \ref{thm1} may be viewed as the starting point for the problem of finding extremal quasiconformal mappings between Kor\'anyi ellipsoidal rings; we shall be concerned with this problem elsewhere.

\section{Preliminaries}\label{sec-prel}
In this section we describe in brief some well known facts about the Heisenberg group. 
For more information, we refer for instance to \cite{KR85,KR95,He-KO98}.

\subsection{The Heisenberg group} 

The (first) Heisenberg group $\mathbb{H}^1$ is the set $\mathbb{C}\times\mathbb{R}$ with multiplication $*$ given by
$$
(z,t)*(w,s)=(z+w,t+s+2\Im(z\overline{w})),
$$
for every $(z,t)$ and $(w,s)$ in $\mathbb{H}^1$. The {\it Kor\'anyi gauge} is defined by
$$
|(z,t)|_{\mathbb{H}^1}=||z|^2-it|^{1/2},
$$
for each $(z,t)\in\mathbb{H}^1$.
Then $\mathbb{H}^1$ can be endowed with the {\it Kor\'anyi-Cygan metric} given by
$$
\mathrm{d}_{\mathbb{H}^1}((z,t),(w,s))=|(z,t)^{-1}*(w,s)|_{\mathbb{H}^1},
$$
which is invariant under the left action of $\mathbb{H}^1$. 

The Heisenberg group $\mathbb{H}^1$ is a 2-step nilpotent Lie group; we consider the basis for the left invariant vector fields of $\mathbb{H}^1$ comprising
$$
X=\frac{\partial}{\partial x}+2y\frac{\partial}{\partial t},\quad
Y=\frac{\partial}{\partial y}-2x\frac{\partial}{\partial t},\quad
T=\frac{\partial}{\partial t}.
$$
Denote $\mathfrak{h}$ to be the Lie algebra of $\mathbb{H}^1$. Then there exists a decomposition
$$
\mathfrak{h}=V^1\oplus V^2,
$$
where
$$
V^1={\rm span}_\mathbb{R}\left\{X,Y\right\}\; \;\rm and\; \;V^2={\rm span}_\mathbb{R}\left\{T\right\}.
$$
The {\it contact structure} of $\mathbb{H}^1$ is induced by the 1-form $\omega$ of $\mathbb{H}^1$ given by
$$
\omega={\rm d}t+2(x{\rm d}y-y{\rm d}x)={\rm d}t+2\Im(\overline{z}{\rm d}z),
$$
where $z=x+iy$. By the contact version of Darboux's Theorem, $\omega$ is the unique 1-form  such that $X,\,Y\in\ker\omega,\,\omega(T)=1$.
For each point $p\in \mathbb{H}^1$, $V^1_p=H_p(\mathbb{H}^1)$ is the {\it horizontal tangent space of $\mathbb{H}^1$ at $p$}.
Considering the relations
$$\langle X,X\rangle=\langle Y,Y\rangle=1,\;
\langle X,Y\rangle=\langle Y,X\rangle=0,
$$
we obtain a sub-Riemannian metric $\langle\cdot,\cdot\rangle$ and denote its induced norm by $|\cdot|_h$.
An absolutely continuous curve $\gamma:[a,b]\to\mathbb{H}^1$ (in the Euclidean sense) with
$
\gamma(\tau)=(\gamma_h(\tau),\gamma_3(\tau))\in\mathbb{C}\times\mathbb{R}
$, $\gamma_h(\tau)=x(\tau)+iy(\tau)$, $\gamma_3(\tau)=t(\tau)$,
is called {\it horizontal curve} if $\dot{\gamma}\in H_{\gamma(\tau)}(\mathbb{H}^1)$ for almost all $\tau\in[a,b]$, in other words,
$$
\dot{t}(\tau)=2y\dot{x}(\tau)-2x\dot{y}(\tau)=-2\Im\left(\overline{z(\tau)}\dot{z}(\tau)\right)
$$
for almost all $\tau\in[a,b]$.
The {\it horizontal length} of a smooth rectifiable curve $\gamma=(\gamma_h,\gamma_3)$ with respect to $|\cdot|_h$ is given by
$$
\ell_h(\gamma)=\int_a^b|\dot{\gamma}_h(\tau)|_h{\rm d}\tau=\int_a^b\left(\langle\dot{\gamma}(\tau),X_{\gamma(\tau)}\rangle^2+
\langle\dot{\gamma}(\tau),Y_{\gamma(\tau)}\rangle^2\right)^{1/2}{\rm d}\tau.
$$
The {\it Carnot-Carath\'eodory} distance $\mathrm{d}_{cc}(p,q)$ between any two points $p,q\in\mathbb{H}^1$ is then the infimum of horizontal lengths of all horizontal curves joining $p,q$. The metrics $\mathrm{d}_{\mathbb{H}^1}$ and $\mathrm{d}_{cc}$ are bi-Lipschitz equivalent; note though that $\mathrm{d}_{cc}$ is a path metric whereas $\mathrm{d}_{\mathbb{H}^1}$ is not.

Conformal mappings of $\mathbb{H}^1$ with respect to both Kor\'anyi and Carnot-Carath\'eodory metrics in the Heisenberg group, comprise the following transformations:
\begin{itemize}
\item Left translations $L_p$, $p\in\mathbb{H}^1$, defined by
$$
L_p(q)=p*q,
$$
for each $q\in\mathbb{H}^1$.
\item Rotations $R_\theta$, $\theta\in\mathbb{R}$, defined by
$$
R_\theta(z,t)=(e^{i\theta}z,t),
$$
for each $(z,t)\in\mathbb{H}^1$.
\item Dilations $D_\delta$, $\delta>0$, defined by
$$
D_\delta(z,t)=(\delta z,\delta^2 t),
$$
for each $q\in\mathbb{H}^1$.
\item Conjugation $j$, defined by
$$
j(z,t)=(\overline{z},-t),
$$
for each $q\in\mathbb{H}^1$.
\end{itemize}

\subsection{Contact and quasiconformal transformations}
A transformation $f:\Omega\to\Omega^\prime$
in $\mathbb{H}^1$ between domains $\Omega$ and $\Omega^\prime$ of $\mathbb{H}^1$ is called a {\it contact transformation }, if it preserves the contact structure, i.e. the induced mapping $f^*$ on the 1-form satisfies
$$
f^*\omega=\lambda\omega,
$$
for some non-vanishing real valued function $\lambda$.

The proof of the following proposition is straightforward and may be found for instance in \cite{KR85}.
\begin{prop}
Left-transformations, rotations, dilations, conjugation and inversion are all contact transformations of $\mathbb{H}^1$.
\end{prop}

Let
$$
Z=\frac{1}{2}(X-iY)=\frac{\partial}{\partial z}+i\overline{z}\frac{\partial}{\partial t},\quad \overline{Z}=\frac{1}{2}(X+iY)=\frac{\partial}{\partial \overline{z}}-iz\frac{\partial}{\partial t}.
$$
Let also $f:\Omega\to\Omega'$ be a $C^2$ orientation preserving diffeomorphism between domains in $\mathbb{H}^1$. We shall write $f=(f_I,f_3)$, $f_I=f_1+if_2)$. The {\it distortion function} $K(f,\cdot)$ of $f$ is defined by
$$
K(f,p)=\frac{|Zf_I(p)|+|\overline{Z}f_I(p)|}{|Zf_I(p)|-|\overline{Z}f_I(p)|},\quad p\in\Omega.
$$
The {\it maximal distortion} $K=K_f$ is then given by
$$
K=\sup_{p\in\Omega}K(f,p).
$$
$f$ is called {\it $K$-quasiconformal} if it is contact and $K\ge 1$. It is well-known that $1$-quasiconformal mappings of $\mathbb{H}^1$ are conformal and comprise elements of ${\rm SU}(2,1)$ acting on $\mathbb{H}^1$, that is, compositions of left translations, rotations, dilations, conjugation and inversion.

\medskip

For further reference we prove:
\begin{lem}\label{lem-L}
For $a,b\in\R_*$, let $L=L_{a,b}:\mathbb{H}^1\to\mathbb{H}^1$ be the linear mapping given by $$L=L(x,y,t)=(ax,by,abt),$$ for all $(x,y,t)\in\mathbb{H}^1$. Then $L$ is a contact transformation.
\end{lem}
\begin{proof}
We have first that $L$ is contact:
\begin{eqnarray*}
L^*\omega&=&{\rm d}(abt)+2ax{\rm d}(by)-2by{\rm d}(ax)\\
&=&ab\;\omega.
\end{eqnarray*}
We now write equivalently
$$
L(z,t)=\left(L_I(z,t),L_3(z,t)\right)=\left(\frac{1}{2}\left[(a+b)z+(a-b)\overline{z}\right],\;abt\right).
$$
Since
$$
ZL_I=\frac{a+b}{2},\quad \overline{Z}L_I=\frac{a-b}{2},
$$
we have
$$
K(L,(z,t))=\frac{|a+b|+|a-b|}{|a+b|-|a-b|}=K_L.
$$
If $a=b$, then $K_L=1$ (and $L$ is simply a dilation possibly composed with a rotation). If $a\neq b$, then $K_L>1$.
\end{proof}

\subsection{Horizontal gradient}
Let $X,Y$ be the left invariant vector fields as above and
suppose that $u:\mH\to\R$ is a $C^1$ function defined in a domain $\Omega\subset\mathbb{H}^1$. Then if $p\in \Omega$, the {\it horizontal gradient of $u$} at $p$  is given by
$$\nabla^h_pu=(X_p (u))X_p+(Y_p(u))Y_p$$
and its norm is
$$|\nabla^h_pu|=\sqrt{(X_p(u))^2+(Y_p(u))^2}.$$
\subsection{Modulus and capacity}\label{sec-modcap}
Let $\omega\subset\mathbb{H}^1$ be a domain in the Heisenberg group and let also $\Gamma$ be a family of locally rectifiable curves in $\Omega$. Denote by ${\rm d}{\mathcal{L}^3}$ the volume element of the usual Lebesgue measure in $\mathbb{C}\times\mathbb{R}$ and by ${\rm d}s$ the horizontal arc length element of a curve $\gamma\in \Gamma$. The {\it modulus of $\Gamma$} is given by
$$
{\rm mod}(\Gamma)=\inf\iiint_{\Omega}\rho^4{\rm d}{\mathcal{L}^3},
$$
where the infimum is taken over all Borel functions $\rho:\; \Omega\to[0,\infty]$ satisfying
\begin{equation}\label{eq0}
  \int_\gamma\rho{\rm d}s\geq1,
\end{equation}
for all $\gamma\in\Gamma$. Functions $\rho$ satisfying (\ref{eq0}) are called \textit{admissible} for $\Gamma$; the set of all admissible functions for $\Gamma$ shall be denoted by
${\rm adm}(\Gamma).
$


Let $E$ and $F$ two nonempty disjoint closed sets in $\mathbb{H}^1$; we assume that $E$ is compact and $0\in E$ and also $\infty\in F$. We denote by $(E,F)$ the domain $\mathbb{H}^1\setminus (E\cup F)$; such a domain shall  be called a {\it condenser}. 
Its \textit{capacity} is given by
\begin{equation}\label{0.1}
  {\rm cap}(E,F)=\inf\int_{\mathbb{H}^1}|\nabla^h u|^4{\rm d}{\mathcal{L}^3};
\end{equation}
here $|\nabla^h u|$ is the norm of the horizontal differential of $u$ and the infimum is taken over all $u$ in $\mathbb{H}^1$ such that $u|E=1$ and $u|F=0$. Such a function $u$ is called \textit{admissible} for the condenser $(E,F)$.

Both modulus of a curve family and capacity of a condenser are conformal invariants: they are left infariant by conformal transformations of the Heisenberg group. In the proof of our theorem we shall make use of the following proposition (see Proposition 2.17 of \cite{He-KO98}):
\begin{prop}\label{pro1}
If $(E,F)$ is a condenser in $\mathbb{H}^1$ and $\Gamma$ is the family of locally smooth rectifiable curves joining its two pieces of the boundary,  then
\begin{equation*}
  {\rm cap}(E,F)={\rm mod}(\Gamma).
\end{equation*}
\end{prop}

\section{The Kor\'anyi ellipsoidal ring and its modulus}\label{sec-main}
In this section we prove Theorem \ref{thm1}; to do so we set up the notation first.
For fixed $A>1$, let
the {\it Kor\'anyi spherical ring}
$$\mathcal{R}=\mathcal{R}_{1,A}=\left\{(x,y,t)\in \mathbb{H}^1:1\leq\Big(x^2+y^2\Big)^2+t^2\leq A^4\right\}.$$
Fix numbers $0<b<a$. We shall consider the linear map $L=L_{a,b}:\mH^1\to\mH^1$ given by
$$L(x,y,t)=(ax,by,abt).$$
By Lemma \ref{lem-L}, $L$ is a contact quasiconformal self-map of $\mH^1$, its distortion function $K=K_L$ is constant and equals to $a/b$. The {\it Kor\'anyi ellipsoidal ring} in the Heisenberg group
$$\mathcal{E}=\mathcal{E}_{1,A,a,b}=\left\{(x,y,t)\in \mathbb{H}^1:1\leq\left(\frac{x^2}{a^2}+\frac{y^2}{b^2}\right)^2+\frac{t^2}{a^2b^2}\leq A^4\right\}$$
is the $L$-image of the Kor\'anyi spherical ring $\mathcal{R}_{1,A}$.

\medskip

To prove our theorem, we shall use Proposition \ref{pro1}; the result will follow after estimating the capacity $\mathrm{cap}(\mathcal{E})$ from above and the modulus ${\rm mod}(\calE)$ from below by the same number, which in our setup is
$$
\left(\frac{3}{8}\Big(K^2+\frac{1}{K^2}\Big)+\frac{1}{4}\right)\frac{\pi^2}{(\log A)^3}.
$$


\medskip

\noindent {\it Proof of Theorem \ref{thm1}.}
The proof is divided into two steps.
\subsubsection*{Step 1. Estimating the capacity}
 In the first step we prove that 
 \begin{equation}\label{cap-est}
{\rm cap}(\calE)\leq\left(\frac{3}{8}\Big(K^2+\frac{1}{K^2}\Big)+\frac{1}{4}\right)\frac{\pi^2}{(\log A)^3}.
\end{equation}
For this purpose, we consider the function $u_0:\mH^1\to\R$ given by
$$u_0(x,y,t)=\frac{\log \left(\left(\frac{x^2}{a^2}+\frac{y^2}{b^2}\right)^2+\frac{t^2}{a^2b^2}\right)}{4\log A}.$$
This is an admissible function, i.e., it satisfies the following:
\begin{itemize}
  \item If $\mathcal{E}_{0,1}$ is the inner boundary of $\calE$,
  $$\mathcal{E}_{0,1}=\left\{(x,y,t)\in \mathbb{H}^1:\left(\frac{x^2}{a^2}+\frac{y^2}{b^2}\right)^2+\frac{t^2}{a^2b^2}=1\right\},$$ then $u_0\equiv 0$ in $\calE_{0,1}$.
  \item If $\mathcal{E}_{0,A}$ is the outer boundary of $\calE$,
  $$\mathcal{E}_{0,A}=\left\{(x,y,t)\in \mathbb{H}^1:\left(\frac{x^2}{a^2}+\frac{y^2}{b^2}\right)^2+\frac{t^2}{a^2b^2}= A^4\right\},$$ then $u_0\equiv 1$ in $\mathcal{E}_{0,A}$.
\end{itemize}
Our aim is to explicitly calculate the following integral:
 $$\mathcal{I}_0=\iiint_{\mathcal{E}}|\nabla^hu_0|^4 \mathrm{d}x\mathrm{d}y\mathrm{d}t.$$
 For this, we introduce Kor\'anyi ellipsoidal coordinates in the Heisenberg group. These coordinates are given by:

 \begin{eqnarray*}
  x & = &ar\sqrt{\sin\varphi}\cos\theta, \\
  y & = &br\sqrt{\sin\varphi}\sin\theta,\\
  t & = &abr^2\cos\varphi,
\end{eqnarray*}
where $r\in[0,+\infty)$, $\varphi\in[0,\pi]$, $\theta\in[0,2\pi]$. The Jacobian determinant is then
$$J(r,\varphi,\theta)=\left|\begin{array}{ccc}
                         a\sqrt{\sin\varphi}\cos\theta & \frac{ar\cos\varphi\cos\theta}{2\sqrt{\sin\varphi}} & -ar\sqrt{\sin\varphi}\sin\theta \\
                        b\sqrt{\sin\varphi}\sin\theta & \frac{br\cos\varphi\sin\theta}{2\sqrt{\sin\varphi}} & br\sqrt{\sin\varphi}\cos\theta \\
                         2abr\cos\varphi & -abr^2\sin\varphi & 0
                       \end{array}
                       \right|
= a^2b^2r^3.
$$
In Kor\'anyi ellipsoidal coordinates the function $u_0$ is given by
$$
u_0(r,\theta,\varphi)=\frac{\log r}{\log A}.
$$
Moreover,
\begin{eqnarray*}
&&
X=\left(r_x+2y r_t\right)\partial_r
+\left(\theta_x+2y \theta_t\right)\partial_\theta
+\left(\varphi_x+2y \varphi_t\right)\partial_\varphi,\\
&&
Y=\left(r_y-2x r_t\right)\partial_r
+\left(\theta_y-2x \theta_t\right)\partial_\theta
+\left(\varphi_y-2x \varphi_t\right)\partial_\varphi.
\end{eqnarray*}
Therefore
\begin{equation*}
Xu_0
=\frac{\sqrt{\sin\varphi}}{ar\log A}\sin\left(\varphi+\theta\right),\quad
Yu_0
=\frac{-\sqrt{\sin\varphi}}{br\log A}\cos\left(\varphi+\theta\right).
\end{equation*}
Let $K=a/b$. Then by using Fubini's Theorem we have
\begin{eqnarray*}
\mathcal{I}_0
&=& \int^{2\pi}_0\left(\int^{\pi}_0\left(\int^A_1\left( \frac{\sin\varphi}{a^2r(\log A)^2}\sin^2(\varphi+\theta)+\frac{\sin\varphi}{b^2r(\log A)^2}\cos^2(\varphi+\theta)\right)^2a^2b^2r^3 \mathrm{d}r\right)\mathrm{d}\varphi\right)\mathrm{d}\theta\\
& =& \frac{1}{(\log A)^3}\int^{2\pi}_0\left(\int^{\pi}_0
\left(\sin\varphi\Big(\frac{1}{K}\sin^2(\varphi+\theta)+K\cos^2(\varphi+\theta)\Big)\right)^2\mathrm{d}
\varphi\right)\mathrm{d}\theta\\
& =&\frac{1}{(\log A)^3}\cdot\mathcal{J}_0.
\end{eqnarray*}
We only have to show that $$\mathcal{J}_0
   = \left(\frac{3}{8}\Big(K^2+\frac{1}{K^2}\Big)+\frac{1}{4}\right){\pi^2}.$$
To do this, we may first expand
\begin{align*}
&  \left(\sin\varphi\Big(\frac{1}{K}\sin^2(\varphi+\theta)+K\cos^2(\varphi+\theta)\Big)\right)^2 \\
& = \frac{1}{K^2}\sin^2\varphi\sin^4(\varphi+\theta)+K^2\sin^2\varphi\cos^4(\varphi+\theta)+2\sin^2\varphi\sin^2(\varphi+\theta)\cos^2(\varphi+\theta).~~~~
\end{align*}
We then calculate separately the integrals
\begin{eqnarray*}
&&
\mathcal{J}_0^1=\int^{2\pi}_0\left(\int^{\pi}_0 \sin^2\varphi\sin^4(\varphi+\theta)\mathrm{d}\varphi\right)\mathrm{d}\theta,\\
&&
\mathcal{J}_0^2=\int^{2\pi}_0\left(\int^{\pi}_0 \sin^2\varphi\cos^4(\varphi+\theta)\mathrm{d}\varphi\right)\mathrm{d}\theta,\\
&&
\mathcal{J}_0^3=\int^{2\pi}_0\left(\int^{\pi}_0 2\sin^2\varphi\sin^2(\varphi+\theta)\cos^2(\varphi+\theta)\mathrm{d}\varphi\right)\mathrm{d}\theta.
\end{eqnarray*}
It is a simple calculus exercise then to show that
$$
\mathcal{J}_0^1=\frac{3\pi^2}{8}=\mathcal{J}_0^2,\quad \mathcal{J}_0^3=\frac{\pi^2}{4}.
$$
We thus have
 $$
 \mathcal{I}_0=\frac{1}{(\log A)^3}\cdot\mathcal{J}_0=\frac{1}{(\log A)^3}\left(K^2\mathcal{J}_0^1+\frac{1}{K^2}\mathcal{J}_0^2+\mathcal{J}_0^3\right)=
 \left(\frac{3}{8}\Big(K^2+\frac{1}{K^2}\Big)+\frac{1}{4}\right)\frac{\pi^2}{(\log A)^3}.
 $$
Since by the definition of capacity,
$$\mathrm{cap}(\mathcal{E})=\inf_{u\;\text{admissible}}\iiint_{\mathcal{E}}|\nabla^hu|^4\mathrm{d}x\mathrm{d}y\mathrm{d}t,$$
we have
$$ 
\mathrm{cap}(\mathcal{E})\leq\mathcal{I}_0
$$
which proves (\ref{cap-est}); this completes the first part of the proof.
\subsubsection*{Step 2. Estimating the modulus}
In the second step we will show that the following holds:
\begin{equation}\label{ineq2}
\left(\frac{3}{8}\Big(K^2+\frac{1}{K^2}\Big)+\frac{1}{4}\right)\frac{\pi^2}{(\log A)^3}\leq \mathrm{mod}(\mathcal{E}).
\end{equation}
For this purpose, we consider the family $\Gamma_0$ comprising the integral curves of the horizontal vector field $\nabla^h u_0$. This family belongs to the greater family $\Gamma$ of horizontal arcs joining the two pieces of the boundary of $\calE$ and therefore
$$
{\rm mod}(\Gamma_0)\le{\rm mod}(\Gamma)={\rm mod}(\calE).
$$ 
If $\gamma\in\Gamma_0$, we shall consider it parametrised by horizontal arc-length,
$$
\gamma(s)=\left(x(s),y(s),t(s)\right),
\,s\in[0,\ell(\gamma)],
$$
and we also let
\begin{equation}\label{xyz}
x(s)=ar(s)\sqrt{\sin\varphi(s)}
\cos\theta(s),\;y(s)=br(s)\sqrt{\sin\varphi(s)}\sin\theta(s),\;
t(s)=abr^2(s)\cos\varphi(s).
\end{equation}
The differential equations of curves in $\Gamma_0$ are then
\begin{align}
&\dot{x}(s)=\frac{\sqrt{\sin\varphi(s)}}{ar\log A}\sin(\varphi(s)+\theta(s));\label{eq3}\\
&\dot{y}(s)=\frac{-\sqrt{\sin\varphi(s)}}{br\log A}\cos(\varphi(s)+\theta(s));\label{eq4}\\
&\dot{t}(s)+ 2x(s)\dot{y}(s)-2y(s)\dot{x}(s)=0,\label{eq5}
\end{align}
with the latter being the horizontality condition.
By differentiating the equations in (\ref{xyz}) we also have
\begin{eqnarray}\label{x}
\dot{x}&= &a\sqrt{\sin\varphi}\cos\theta\,\dot{r}+\frac{ar}{2\sqrt{\sin\varphi}}\cos\varphi\cos\theta\,\dot{\varphi}-ar\sqrt{\sin\varphi}\sin\theta\,\dot{\theta},\\
\label{y}
\dot{y}&=&b\sqrt{\sin\varphi}\cos\theta\,\dot{r}+\frac{br}{2\sqrt{\sin\varphi}}\cos\varphi\sin\theta\,\dot{\varphi}+br\sqrt{\sin\varphi}\cos\theta\,\dot{\theta},\\
\label{z}
\dot{t}&=&2ab\cos\varphi\, r\dot{r}-abr^2\sin\varphi\,\dot{\varphi}.
\end{eqnarray}
Our aim is to show that
\begin{equation}\label{rr}
r\dot{r} = \frac{\sin\varphi}{\log A}\left(\frac{\sin^2(\varphi+\theta)}{a^2}+\frac{\cos^2(\varphi+\theta)}{b^2}\right)
\end{equation}
holds.
By equating the right hand-sides of (\ref{eq3}), (\ref{x}) and the right hand-sides of (\ref{eq4}), (\ref{y}) we respectively have
\begin{eqnarray}
\sin\varphi\cos\theta\,r\dot{r}+\frac{1}{2}\cos\varphi\cos\theta\,r^2\dot{\varphi}-\sin\varphi\sin\theta\,r^2\dot{\theta} &=&  \frac{\sqrt{\sin\varphi}}{a^2\log A}\sin(\varphi+\theta),\label{eq6}\\
\sin\varphi\sin\theta\,r\dot{r}+\frac{1}{2}\cos\varphi\sin\theta\,r^2\dot{\varphi}-\sin\varphi\cos\theta\,r^2\dot{\theta} &=& \frac{-\sqrt{\sin\varphi}}{b^2\log A}\cos(\varphi+\theta).\label{eq7}
\end{eqnarray}
On the other hand, from (\ref{eq5}) and (\ref{z}) we deduce
$$
 -2x\dot{y}+2y\dot{x}=2ab\cos\varphi\, r\dot{r}-abr^2\sin\varphi\,\dot{\varphi}.
$$
Since
\begin{eqnarray*}
x\dot y
&=& ab\sin\varphi\sin\theta\cos\theta\,r\dot{r}+\frac{ab}{2}\cos\varphi\sin\theta\cos\theta\,r^2\dot{\varphi}
+ab\sin\varphi\cos^2\theta\,r^2\dot{\theta,}\\
y\dot{x}
&=& ab\sin\varphi\sin\theta\cos\theta\,r\dot{r}+\frac{ab}{2}\cos\varphi\sin\theta\cos\theta\,r^2\dot{\varphi}-ab\sin\varphi\sin^2\theta\,r^2\dot{\theta},
\end{eqnarray*}
we finally have
\begin{equation}\label{eq8}
\cos\varphi\,r\dot{r}-\frac{1}{2}\sin\varphi\,r^2\dot{\varphi}+ \sin\varphi\,r^2\dot{\theta}=0,
\end{equation}
and our desired equation (\ref{rr}) follows by solving the system of (\ref{eq6}),(\ref{eq7})  and (\ref{eq8}).
Now, for an arbitrary $\gamma\in\Gamma_0$ we have that its horizontal norm is
\begin{align*}
|\dot{\gamma}|_h&=\sqrt{\dot{x}^2(s)+\dot{y}^2(s)}\\
&=\sqrt{\frac{1}{r^2}\left(\frac{\sin\varphi}{\log^2 A}\Big(\frac{\sin^2(\varphi+\theta)}{a^2}+\frac{\sin^2(\varphi+\theta)}{b^2}\Big)\right)}\\
&=\sqrt{\frac{1}{r^2\log A}r\dot{r}}\\
&=\Big(\frac{1}{\log A}\Big)^{1/2}\Big(\frac{\dot{r}}{r}\Big)^{1/2},
\end{align*}
where we have used (\ref{rr}).
Thus if $\rho\in \mathrm{adm}(\Gamma_0)$ we have
$$
1\leq\frac{1}{\log^{1/2} A}\int^{\ell(\gamma)}_0\rho(\gamma(s))\left(\frac{\dot{r}(s)}{r(s)}\right)^{1/2}\mathrm{d}s
$$
Let $r=r(s)$ and $K=a/b$. Then $\mathrm{d}r=\dot{r}(s)\mathrm{d}s$ and by using again (\ref{rr}) the above inequality is written as
\begin{align*}
1
&\leq\int^A_1\frac{ab\rho(r,\varphi,\theta)}{\sqrt{\sin\varphi\left(b^2\sin^2(\varphi+\theta)+a^2\cos^2(\varphi+\theta)\right)}}\mathrm{d}r\\
&=\int^A_1\rho(r,\varphi,\theta)(ab)^{1/2}r^{3/4}\cdot\frac{(ab)^{1/2}}{r^{3/4}\sqrt{\sin\varphi\left(b^2\sin^2(\varphi+\theta)+a^2\cos^2(\varphi+\theta)\right)}}\mathrm{d}r\\
&=\int^A_1\rho(r,\varphi,\theta)(ab)^{1/2}r^{3/4}\cdot\frac{1}{r^{3/4}\sqrt{\sin\varphi\left(\frac{1}{K}\sin^2(\varphi+\theta)+K\cos^2(\varphi+\theta)\right)}}\mathrm{d}r.
\end{align*}
We apply H\"older's inequality to obtain
$$
1\leq\left(\int^A_1\rho^4\left(r,\varphi,\theta\right)a^2b^2r^3\mathrm{d}r\right)^{1/4}\left(\int^A_1\left(\frac{1}{\sqrt{\sin\varphi\Big(\frac{1}{K}\sin^2(\varphi+\theta)
+K\cos^2(\varphi+\theta)\Big)}}\right)^{4/3}\frac{\mathrm{d}r}{r}\right)^{3/4}.
$$
This implies
$$
\Big(\frac{1}{\log A}\Big)^{3/4}\sqrt{\sin\varphi\Big(\frac{1}{K}\sin^2(\varphi+\theta)
+K\cos^2(\varphi+\theta)\Big)}\leq\Big(\int^A_1\rho^4(r,\varphi,\theta)a^2b^2r^3\mathrm{d}r\Big)^{1/4}.
$$
Take both sides of the above inequality in the fourth power; then
$$
\frac{1}{\log^3 A}\sin^2\varphi\left(\frac{1}{K}\sin^2(\varphi+\theta)
+K\cos^2(\varphi+\theta)\right)^2\leq\int^A_1\rho^4(r,\varphi,\theta)a^2b^2r^3\mathrm{d}r.
$$
By integrating with respect to $\varphi$ and $\theta$ we have after changing the coordinates that
$$
\left(\frac{3}{8}\Big(K^2+\frac{1}{K^2}\Big)+\frac{1}{4}\right)\frac{\pi^2}{(\log A)^3}\leq \iiint_{\mathcal{E}}\rho^4(x,y,t)\mathrm{d}x\mathrm{d}y\mathrm{d}t.
$$
Finally, by taking the infimum over all $\rho\in \mathrm{adm}(\Gamma_0)$ we obtain
$$
\left(\frac{3}{8}\Big(K^2+\frac{1}{K^2}\Big)+\frac{1}{4}\right)\frac{\pi^2}{(\log A)^3}\leq \mathrm{mod}(\Gamma_0)\leq \mathrm{mod}(\mathcal{E})=\mathrm{cap}(\mathcal{E}),
$$
where the last equality holds due to Proposition \ref{pro1}.
This, combined with (\ref{cap-est}) completes the proof.\qed



\end{document}